\documentclass[12pt]{amsart}

\usepackage[top=1in, bottom=1in, left=1in, right=1in]{geometry}
\usepackage{times}

\usepackage{amssymb,amsmath,amsthm,mathtools,mathrsfs}
\usepackage{enumitem}

\setlist[1]{itemsep=0.5em, topsep=0.5em}

\usepackage[table,dvipsnames]{xcolor}
\usepackage{color}
\definecolor{red}{rgb}{1,0,0}
\definecolor{orange}{rgb}{0.7,0.3,0}
\definecolor{blue}{rgb}{0,.3,.7}
\definecolor{green}{rgb}{0,.6,.4}

\PassOptionsToPackage{hyphens}{url}\usepackage[colorlinks=true, linkcolor=NavyBlue, citecolor=teal, urlcolor=gray]{hyperref}   

\renewcommand{\le}{\leqslant}
\renewcommand{\leq}{\leqslant}
\renewcommand{\ge}{\geqslant}
\renewcommand{\geq}{\geqslant}{}

\numberwithin{equation}{section}

\theoremstyle{plain}
\newtheorem{theorem}{Theorem}
\newtheorem{cor}{Corollary}[section]
\newtheorem{corollary}[cor]{Corollary}

\newtheorem{lemma}[cor]{Lemma}

\newtheorem{proposition}[cor]{Proposition}

\theoremstyle{remark}
\newtheorem{rem}{Remark}
\newtheorem*{rem*}{Remark}

\theoremstyle{definition}

\newcommand{\N}{\mathbb{N}}
\newcommand{\Z}{\mathbb{Z}}

\newcommand{\R}{\mathbb{R}}

\newcommand{\CE}{\mathcal{E}}

\newcommand{\CS}{\mathcal{S}}

\newcommand{\nn}{\nonumber \\}

\newcommand{\dee}{\,\mathrm{d}}

\newcommand{\Log}{\operatorname{Log}}

\newcommand{\fl}[1]{\left\lfloor#1\right\rfloor}

\newcommand{\eps}{\varepsilon}
\renewcommand{\phi}{\varphi}
\renewcommand{\rho}{\varrho}

\newcommand{\bs}\boldsymbol{}

\renewcommand{\mod}[1]{\,(\mathrm{mod}\,#1)}

\newcommand{\bg}{\big}
\newcommand{\bgg}{\Big}
\newcommand{\bggg}{\bigg}

\definecolor{red}{rgb}{1,0,0}

\definecolor{orange}{rgb}{0.7,0.3,0}

\definecolor{blue}{rgb}{.2,.6,.75}

\definecolor{green}{rgb}{.4,.7,.4}

\begin{document}

\title{An upper bound on the mean value of the Erd\H os--Hooley Delta function}

\author{Dimitris Koukoulopoulos}
\address{D\'epartement de math\'ematiques et de statistique\\
Universit\'e de Montr\'eal\\
CP 6128 succ. Centre-Ville\\
Montr\'eal, QC H3C 3J7\\
Canada}
\email{{\tt dimitris.koukoulopoulos@umontreal.ca}}

\author{Terence Tao}
\address{Department of Mathematics\\
UCLA \\
405 Hilgard Ave\\
Los Angeles, CA 90095\\
USA}
\email{{\tt tao@math.ucla.edu}}

\subjclass[2020]{Primary: 11N25; Secondary: 11N37, 11N64}
\keywords{Erd\H os--Hooley Delta function, divisors of integers, concentration function, method of moments}

\date{\today}

\begin{abstract}
The Erd\H os--Hooley Delta function is defined for $n\in\mathbb{N}$ as  $\Delta(n)=\sup_{u\in\mathbb{R}} \#\{d|n : e^u<d\le e^{u+1}\}$. We prove that $\sum_{n\le x} \Delta(n) \ll x(\log\log x)^{11/4}$ for all $x\ge100$. This improves on earlier work of Hooley, Hall--Tenenbaum and La Bretèche--Tenenbaum.
\end{abstract}

\maketitle

\section{Introduction}

The \emph{Erd\H os--Hooley Delta function} (\href{https://oeis.org/A226898}{oeis.org/A226898}) is defined for a natural number $n$ as
\[
\Delta(n) \coloneqq \sup_{u \in \R} \#\{d|n: e^u<d\le e^{u+1}\} .
\]
Erd\H os introduced this function in the 1970s \cite{erdos1,erdos2} and studied certain aspects of its distribution in joint work with Nicolas \cite{erdos-nicolas,erdos-nicolas2}. However, it was not until the work of Hooley in 1979 that $\Delta$ was studied in more detail \cite{hooley}. Specifically, Hooley proved that
\begin{equation}
	\label{eq:hooley's estimate}
	\sum_{n\le x} \Delta(n) \ll x (\Log x)^{\frac{4}{\pi}-1} 
\end{equation}
for any $x\ge1$. Here and in the sequel we use the notation
\[
\Log x \coloneqq \max\{1, \log x\} \quad\text{for}\ x>0,
\]
and also define
\[
\Log_2 x \coloneqq \Log(\Log x)
\quad\text{and}\quad 
\Log_3 x \coloneqq \Log(\Log_2 x);
\]
see also Section \ref{notation-sec} below for our asymptotic notation conventions.

To put Hooley's estimate \eqref{eq:hooley's estimate}  into context, let us note that $1\le \Delta\le \tau$ with $\tau(n)=\#\{d|n\}$ the divisor function. Thus we have the trivial bounds
\begin{equation}
	\label{eq:hooley-trivial}
	x\ll \sum_{n\le x} \Delta(n) \ll x\Log x 
\end{equation}
for $x\ge1$. Comparing \eqref{eq:hooley's estimate} with \eqref{eq:hooley-trivial}, we see that $\Delta$ is on average of genuinely smaller order than $\tau$. This savings is crucial: as Hooley demonstrated (see \cite{hooley,tenenbaum-delta-applications}, and Remarks \ref{rem:diophantine} and \ref{rem-hooley} below), it can be exploited to count solutions to certain Diophantine equations that are not amenable to more ``standard'' techniques, as well as to improve bounds on certain Diophantine approximation results.

In a series of papers, Hall and Tenenbaum improved significantly Hooley's estimate for $\Delta$ and for various generalizations of it; see \cite{HT1,HT2,HT3}, and also \cite{HT-book}. Their work culminated in the following estimates \cite[Theorems 60 and 70]{HT-book}: for every fixed $\eps>0$ and for every $x\ge1$, we have
\begin{equation}
	\label{eq:HT}
	x\Log_2 x\ll \sum_{n\le x}\Delta(n) \ll_\eps x \exp\bgg( \bg(\sqrt{2}+\eps \bg)\sqrt{\Log_2x\Log_3x} \, \bgg) .
\end{equation}
The upper bound was improved recently by La Bret\`eche and Tenenbaum \cite{breteche} to
\[
\sum_{n \leq x} \Delta(n) \ll_\eps x \exp\bgg( \bg(\sqrt{2}\log 2+\eps \bg) \sqrt{\Log_2 x}\, \bgg)   
\]
for every fixed $\eps>0$ and for every $x\ge1$.

The main result of this note is the following further sharpening of the upper bound.

\begin{theorem}[Mean value bound]\label{main}  For $x\ge1$, we have
\[
\sum_{n \le x} \Delta(n) \ll x (\Log_2 x)^{11/4} .
\]
\end{theorem}

\begin{rem}
The average value of $\Delta$ is dominated by ``atypical'' integers. Indeed, we know from results in \cite{breteche} and in \cite{fgk} that, for every fixed $\eps>0$, we have
\[
(\Log_2 x)^{\eta - \eps} \le \Delta(n) \le (\Log_2 x)^{\theta + \eps}
\]
for all but $o(x)$ integers $n\in[1,x]$, where $\theta\coloneqq \frac{\log2}{\log2+1/\log2-1}=0.6102\dots$ and $\eta=0.3533\dots$ is another constant\footnote{The precise definition is $\eta=(\log2)/\log(2/\rho)$, where $\rho$ is the unique number in $[0,1/3]$ satisfying the equation $1 - \rho/2 = \lim_{j\to\infty} 2^{j-2}/\log a_j$	with $a_1 = 2$, $a_2 = 2 + 2^{\rho}$ and $a_j = a_{j-1}^2 + a_{j-1}^{\rho} - a_{j-2}^{2\rho}$ for $j\in\Z_{\ge3}$.}. 
However, the leftmost inequality in \eqref{eq:HT} implies that the mean value of $\Delta(n)$ over $n\in[1,x]$ is of larger order. As a matter of fact, it appears that the average value has significant contributions from integers for which $\Delta(n)$ is as large as $(\log x)^{\log4-1}$. Indeed, in a recent preprint of Kevin Ford and the two authors of the present paper \cite{FKT}, it was shown that 
\[
\sum_{n \leq x} \Delta(n) \gg x(\Log_2x)^{1+\eta},
\]
with $\eta$ as above. Ignoring factors of $(\Log_2x)^{O(1)}$, this paper shows roughly that for any choice of  $\Log_2y\in[\eps\Log_2x,(1-\eps)\Log_2x]$, we have $\Delta(n) \gtrapprox (\Log y)^{\log4-1}$ for $\gtrapprox x/(\Log y)^{\log4-1}$ integers $n\le x$ with $\omega(n)=\Log_2y+\Log_2x+O(1)$ (those that have about $2\Log_2y$ prime factors $\le y$, and about $\Log_2x-\Log_2y$ prime factors in $(y,x]$).
\end{rem}

\begin{rem}\label{rem:diophantine}
As indicated before, estimates on the partial sums of the $\Delta$-function have applications to counting solutions to certain Diophantine equations. In \cite{robert}, Olivier Robert studied the following question: given integers $k\ge2$, $\ell_k\ge\cdots \ge\ell_1\ge2$ and $c_0,c_1,\dots,c_k\ge1$ such that $\sum_{j=1}^k 1/\ell_j=1/2$, let $S^{\neq}(x)$ denote the number of tuples $(m_0,m_1,\dots,m_k,n_0,n_1,\dots,n_k)\in\N^{2k+2}$ such that
\begin{equation}
	\label{eq:robert diophantine equation}
		c_0m_0^2+ \sum_{j=1}^k c_j m_j^{\ell_j} = c_0n_0^2+ \sum_{j=1}^k c_j n_j^{\ell_j} \le x.
\end{equation}
A straightforward adaptation of \cite{robert} leads to the estimate
	\begin{equation}
		\label{eq:robert improvement}	
			S^{\neq} (x) \ll x(\Log_2x)^{2^{4L}+15/4}
	\end{equation}
with $L=\max\{\ell_1,\dots,\ell_k\}$ and the implied constant depending at most on parameters $k$, $c_1,\dots,c_k$ and $\ell_1,\dots,\ell_k$, which improves Theorem 1.1 of \cite{robert}. In turn, this leads to a similar improvement of Theorem 1.2 of \cite{robert}. We will outline the proof of \eqref{eq:robert improvement} in Section \ref{sec:robert improvement}. 
\end{rem}

\begin{rem}\label{rem:erdos}
	Theorem \ref{main} has applications to a problem of Erd\H os on sets whose subset sums are not squares. Specifically, assume that $c$ is a constant such that
	\begin{equation}
			\label{eq:Delta bound assumption}
		\sum_{n\le x}\Delta(n) \ll x(\Log_2x)^c 
		\quad\text{for all}\ x\ge1. 		
	\end{equation}
	In an upcoming paper, David Conlon, Jacob Fox and Huy Pham develop a new combinatorial argument that deduces from \eqref{eq:Delta bound assumption} that any subset $A$ of $\{1,2,\dots,N\}$ with $|A| \geq N^{1/3} (\Log_2 N)^{c'}$ for some appropriate $c'=c'(c)$ has the property that its set of subset sums $\{\sum_{b\in B}b: B\subseteq A\}$ contains a square. This improves the earlier bound of $N^{1/3} (\Log N)^C$ with $C>0$ of Nguyen and Vu \cite{nguyen-vu}.
\end{rem}

\begin{rem}\label{rem-hooley}  In \cite{hooley}, Hooley used the bound \eqref{eq:hooley's estimate} to show that for any irrational $\theta$ and real $\gamma$, and any $\eps>0$, the inequality $\| n^2 \theta - \gamma \| \leq n^{-1/2}( \log n)^{\frac{2}{\pi}-\frac{1}{2}+\eps}$ holds for infinitely many $n$, where $\|x\|$ denotes the distance of the real number $x$ from the nearest integer. Tenenbaum \cite{tenenbaum-delta-applications} improved the logarithmic factor in this bound using \eqref{eq:HT}. Similarly, it should be possible to use Theorem \ref{main} to improve further the logarithmic factor, but we will not pursue this matter here.  In the homogeneous case $\gamma=0$, the more significant improvement $\| n^2 \theta \| \leq n^{-2/3+\eps}$ was achieved (for arbitrary real $\theta$) by Zaharescu \cite{zah}.
\end{rem}

\subsection*{Acknowledgments} 
The authors would like to thank Huy Pham for bringing to their attention the connection mentioned in Remark \ref{rem:erdos} above. They would also like to thank Kevin Ford, Olivier Robert, Alexandru Zaharescu, and an anonymous referee for their careful reading and useful remarks on earlier versions of the paper.

DK is supported by the Courtois Chair II in fundamental research, by the Natural Sciences and Engineering Research Council of Canada (RGPIN-2018-05699) and by the Fonds de recherche du Qu\'ebec - Nature et technologies (2022-PR-300951).

TT is supported by the National Science Foundation grant DMS-1764034 and by a Simons Investigator Award.

\smallskip

DK dedicates this paper to his son Paris Christopher, who rested in his arms as a newborn many sleepless nights during the writing of the paper.

\section{Notation}\label{notation-sec}

We use $X \ll Y$, $Y \gg X$, or $X = O(Y)$ to denote a bound of the form $|X| \leq CY$ for a constant $C$.  If we need this constant to depend on parameters, we indicate this by subscripts, for instance $X \ll_k Y$ denotes a bound of the form $|X| \leq C_k Y$ where $C_k$ can depend on $k$.  We also write $X \asymp Y$ for $X \ll Y \ll X$.  All sums will be over natural numbers unless the variable is $p$, in which case the sum will be over primes. We use $1_E$ to denote the indicator of a statement $E$, thus $1_E$ equals $1$ when $E$ is true and $0$ otherwise.

Given an integer $n$, we write $\tau(n)\coloneqq\sum_{d|n}1$ for its divisor-function and $\omega(n)\coloneqq\sum_{p|n}1$ for the number of its distinct prime factors.

It will be convenient, for each $x \geq 1$, to work with the set $\CS_{<x}$ denote the set of square-free numbers, all of whose prime factors $p$ are such that $p<x$.  Observe that if $1 \leq y \leq x$, then every $n \in \CS_{<x}$ has a unique factorization $n = n_{<y} n_{\geq y}$, where $n_{<y} \in \CS_{<y}$ and $n_{\geq y}$ lies in the set $\CS_{[y,x)}$ of square-free numbers, all of whose prime factors $p$ are in the interval $[y,x)$.

\section{Methods of proof}

Similarly to other authors, we shall work with logarithmic weights. Specifically, for all $x\ge1$, we have \cite[Theorem 61]{HT-book}
\begin{equation}
	\label{eq:switching to log weights}
	\sum_{n\le x} \Delta(n) \ll \frac{x}{\Log x} \sum_{n\in\CS_{<x}} \frac{\Delta(n)}{n} .	
\end{equation}

Now, for each $u\in\R$, let us define
\begin{equation}\label{deltanu-def}
	\Delta(n; u) \coloneqq \#\{d|n: e^u < d \leq e^{u+1} \} ,
\end{equation}
so that 
\[
\Delta(n)=\sup_{u\in\R} \Delta(n;u).
\]
As with previous work, we introduce the moments
\begin{equation}\label{Mq-def}
 M_q(n) \coloneqq \int_\R \Delta(n;u)^q \dee u
\end{equation}
for $q \geq 1$. Thus, for instance,
\[
M_1(n) = \tau(n)
\]
and
\begin{equation}\label{deltan}
 \Delta(n) = \lim_{q \to \infty} M_q(n)^{1/q}.
\end{equation}
In view of \eqref{deltan}, it is then natural to try to control $M_q(n)$ for large $q$, keeping track of the dependence of constants on $q$. In order to exploit the multiplicative nature of $\Delta$, we employ the identity
\[
\Delta(np;u) = \Delta(n;u) + \Delta(n; u-\log p)
\]
whenever $n$ is a natural number, $p$ is a prime not dividing $n$, and $u$ is a real. Taking the $q^{\mathrm{th}}$ moments of both sides of this identity, we obtain
\[
M_q(pn) = \sum_{\substack{a+b=q \\ 0 \leq b \leq q}} \binom{q}{a} \int_\R \Delta(n;u)^a \Delta(n;u-\log p)^b\dee u.
\]
Extracting out the extreme terms with $b\in\{0,q\}$, we can write this as
\begin{equation}\label{mq0}
 M_q(pn) = 2M_q(n) + \sum_{\substack{a+b=q \\ 1 \leq b \leq q-1}} \binom{q}{a} \int_\R \Delta(n;u)^a \Delta(n;u-\log p)^b\dee u.
\end{equation}
By the use of H\"older's inequality and other tools, one can use this identity to recursively control expressions such as
\[
\sum_{\substack{ n \ge1 \\ \omega(n) \geq k}} 
	\frac{M_{q}(n)^{1/q}}{n^\sigma}
\]
for various $\sigma>1$ and $k \geq 1$, where $\omega(n)$ denotes the number of distinct prime factors of $n$.  See for instance \cite{breteche} for an example of this approach.

\medskip

In our work, we use a variation of the above ideas. Our main guiding heuristic is that $\Delta(n)$ behaves roughly as
\begin{equation}
	\label{eq:delta-heuristic-size}
	\max_{y\in[1,x]} \frac{\tau(n_{<y})}{\Log y}
\end{equation}
for integers $n\in[1,x]$. To give some support to this heuristic, let us note that 
\[
\tau(a) = M_1(a) = \int_{-1}^{\log a} \Delta(a;u) \dee u \le (1+\log a) \Delta(a) 
\]
for any $a\in\N$. Applying this with $a=n_{<y}$ and noticing that $\Delta(n_{<y})\le \Delta(n)$ and that $\log n_{<y}$ is typically of size $\Log y$, we find that the expression in \eqref{eq:delta-heuristic-size} is morally a lower bound (up to constants) for $\Delta(n)$. 

Motivated by the discussion of the above paragraph, we introduce certain sets that are meant to act roughly as level sets of the $\Delta$-function. Precisely, given a parameter $A\ge1$, we define $\tilde{\CS}^A_{<x}$ to be the set  of integers $n\in\CS_{<x}$ such that
\[
\tau(n_{<y}) \le A \Log y \qquad\text{for all}\ y\in[1,x].
\]
Using a simple Markov inequality, we may show that a proportion of $1-O(1/A)$ integers in $\CS_{<x}$ lie also in $\tilde{\CS}_{<x}^A$. As a matter of fact, using a more careful analysis, the same statement holds if we replace $\tilde{\CS}_{<x}^A$ by the set $\CS_{<x}^A$ of integers $n\in\CS_{<x}$ such that
\begin{equation}
		\label{eq:def of S1}
			\tau(n_{<y}) \le A e^{-f_A(y)} \Log y  \qquad\text{for all}\ y\in[1,x],	
\end{equation}
where $e^{-f_A(y)}$ is a Gaussian-type weight concentrated around the region 
\[
\Log_2 y = \frac{\Log A + O\big(\sqrt{\log A}\,\big)}{\log 4 - 1}
\] 
(cf.~Proposition \ref{gauss}). 

Our goal would then be to also show that $\Delta(n)\lessapprox A$ for most $n\in \CS^A_{<x}$. (In fact, we will only be able to show a weaker version of this, which is why the exponent in Theorem \ref{main} is larger than in the lower bound of \eqref{eq:HT}.) In order to achieve this goal, we use \eqref{mq0} and a recursive argument that allows us to control averages of $M_q(n)$ when $n$ ranges over $\CS_{<x}^{q-1,A}$, defined to be the set of $n\in\CS_{<x}^A$ such that 
\begin{equation}
	\label{eq:def of S^(q-1)}
	M_j(n)\le \tau(n)\cdot m_{j,A} \quad\text{for}\ j=2,3,\dots,q-1,
\end{equation}
where the $m_{j,A}$'s are certain suitable quantities growing roughly like $(jA)^j (\log A)^{3j/4}$.

It is important to note that our recursive argument makes use of the following simple but crucial observation: the integral 
\begin{equation}
	\label{eq:symmetry}	
	\binom{q}{a} \int_\R \Delta(n;u)^a \Delta(n;u-\log p)^b\dee u
\end{equation}
is \emph{symmetric} in $a,b$. Indeed, we have $\Delta(n;v) = \Delta(n;\log n-v-1)$ for all but finitely many values of $v\in\R$, because $d\in(e^v,e^{v+1}]$ if and only if $n/d\in[e^{\log n-v-1},e^{\log n-v})$. Thus
\begin{align*}
		\int_\R \Delta(n;u)^a \Delta(n;u-\log p)^b\dee u
		&= \int_\R \Delta(n;\log n-u-1)^a \Delta(n;\log n - u  -1 + \log p  )^b\dee u\\
		&=\int_\R \Delta(n;v-\log p)^a \Delta(n;v)^b\dee v.
\end{align*}
This proves our claim that the integral in \eqref{eq:symmetry} is symmetric in $a,b$.

Now, combining \eqref{mq0} with the symmetry of \eqref{eq:symmetry}, we have the inequality
\begin{equation}\label{mq}
 M_q(pn) \leq 2M_q(n) + 2 \sum_{\substack{a+b=q \\ 1 \leq b \leq q/2}} \binom{q}{a} \int_\R \Delta(n;u)^a \Delta(n;u-\log p)^b\ du.
\end{equation}
To eliminate the factors of $2$ we observe that $\tau(pn)=2\tau(n)$ (recall that $p\nmid n$ here), and hence
\begin{equation}\label{mq-2}
 \frac{M_q(pn)}{\tau(pn)} \leq \frac{M_q(n)}{\tau(n)} + \frac{1}{\tau(n)} \sum_{\substack{a+b=q \\ 1 \leq b \leq q/2}} \binom{q}{a} \int_\R \Delta(n;u)^a \Delta(n;u-\log p)^b\ du.
\end{equation}
We then can apply H\"older's inequality (treating the $\Delta(n;u)^a$ and $\Delta(n;u-\log p)^b$ terms differently) to \eqref{mq-2}, and use our pointwise bounds \eqref{eq:def of S1} and \eqref{eq:def of S^(q-1)},  which will allow us to inductively obtain efficient estimates for the sum
\[
\sum_{n \in \CS_{<x}^{q-1,A}} \frac{M_q(n)/\tau(n)}{n},
\]
where $q \geq 1$, $A \geq 1$, $x\geq 1$ are parameters.

\section{Basic estimates}\label{basic-estimates-sec}

We record here a couple of simple lemmas for easy reference, starting with the following standard consequence of Mertens' theorem:

\begin{lemma}[Mertens' theorem estimate]\label{mertens}  Fix $k\ge0$. For $x\ge y\ge 1$, we have
\[
\sum_{n \in \CS_{[y,x)} } \frac{\tau^k(n)}{n} = \prod_{y\le p<x} \left(1 + \frac{2^k}{p}\right) \asymp_k \bggg(\frac{\Log x}{\Log y}\bggg)^{2^k} .
\]
\end{lemma}

\begin{proof} We have
\[
\log \prod_{y\le p<x} \left(1 + \frac{2^k}{p}\right) = \sum_{y\le p<x} \frac{2^k}{p} + O_k(1) ,
\]
so the lemma follows by a classical estimate of Mertens \cite[Theorem 3.4(b)]{dk-book}.
\end{proof}

We also note the following estimate:

\begin{lemma}[Brun--Titchmarsh inequality]\label{bt}  For $z\ge y \geq z/100\ge1$, we have
\[
\sum_{y\le p\le z} \frac{1}{p} \ll \frac{\log(z/y)}{\log y} + \frac{1}{y^{1/2}}.
\]
\end{lemma}

\begin{proof} Note that $\log(z/y)\asymp (z-y)/y$ and that $1/p\asymp 1/y$ for all primes $p\in[y,z]\subseteq[y,100y]$. Hence, it suffices to show that 
	\begin{equation}
		\label{eq:BT}
		\#\{y\le p\le z\} \ll \frac{z-y}{\log y} + y^{1/2}. 
	\end{equation}
If $z\le y+y^{1/2}$, there are at most $y^{1/2}$ primes in $[y,z]$. On the other hand, if $100y\ge z>y+y^{1/2}$, then \eqref{eq:BT} follows from the Brun--Titchmarsh inequality (see e.g., \cite[Theorem 20.1]{dk-book})
\end{proof}

\section{Control on the divisor function}

Let $x >1$. Let us recall our heuristic argument that $\Delta(n)$ behaves like $\max_{y\in[1,x]}\big( \tau(n_{<y})/\Log y\big)$ for integers $n\in\CS_{<x}$. Our ultimate goal is understand the probability that $\Delta(n)>A$. Motivated by our heuristic, we first study the probability of the event that $\max_{y\in[1,x]}\big(\tau(n_{<y})/\Log y\big)>A$. Equivalently, this is the event that there exists some $y\in[1,x]$ such that $\tau(n_{<y})>A\Log y$. From Mertens' theorem we have
\[
\sum_{n \in \CS_{<x}} \frac{\tau(n_{<y})}{n} = \prod_{p<y} \left( 1 + \frac{2}{p} \right)\prod_{y\le p<x} \left(1+\frac{1}{p}\right) \asymp (\Log x)(\Log y),
\]
and hence by Markov's inequality we see that $\tau(n_{<y}) \leq A \Log y$ for all $n \in \CS_{<x}$ outside of an exceptional set $\CE_{A,y}$ with
\begin{equation}
	\label{eq:naive markov bound}
	\sum_{n \in \CE_{A,y}} \frac{1}{n} \ll \frac{\Log x}{A}  .
\end{equation}

We now give a refinement of this simple analysis, in which we have a single exceptional set that covers all $y\in[1,x]$, and furthermore there is an additional Gaussian-type decay outside of the critical regime $\Log_2 y = \frac{\Log A+O(\sqrt{\Log A})}{\log 4-1}$.

\begin{proposition}\label{gauss}  Let $A \geq 1$.
For any $x > 1$, let $\CS_{<x}^A$ denote the collection of all $n \in \CS_{<x}$ such that
\begin{equation}\label{mjn}
 \tau(n_{<y}) \leq A e^{-f_A(y)} \Log y
\qquad\text{for all}\ y\in[1,x],
\end{equation}
where
\begin{equation}\label{fadef}
f_A(y) \coloneqq \delta \min\bggg\{ \frac{ \bg(\Log_2 y-\frac{\Log A}{\log4-1} \bg)^2 }{\Log A}, \Log A + \Log_2 y\bggg\} 
\end{equation}
and $\delta>0$ is a sufficiently small absolute constant.  Then
\begin{equation}
	\label{eq:prop gauss conclusion}
	\sum_{n \in \CS_{<x} \backslash \CS_{<x}^A} \frac{1}{n} \ll \frac{\Log x}{A} .
\end{equation}
\end{proposition}

\begin{rem*}
The upper bound \eqref{eq:prop gauss conclusion} is sharp. When $\Log_2y=\frac{\Log A}{\log4-1}$, relation \eqref{mjn} becomes $\tau(n_{<y})\le (\log y)^{\log4}$ or, equivalently, $\omega(n_{<y})\le2\Log_2y$. This event occurs with probability roughly equal to $1- (\log y)^{-(\log4-1)}=1-1/A$. A more refined analysis, that uses appropriately adapted results of Ford \cite{ford:smirnov and primes} can show that the left-hand side of \eqref{eq:prop gauss conclusion} is $\asymp\frac{\Log x}{A}$. Hence, the naive Markov bound \eqref{eq:naive markov bound} is actually close to the truth in the critical range of $y$.
\end{rem*}

\begin{proof}  We may assume that $A$ is large, as the claim is immediate from Mertens' inequality otherwise.

Suppose $n \in \CS_{<x} \backslash \CS_{<x}^A$.  Then there exists $y_0\in[1,x]$ such that
\[
\tau(n_{<y_0}) > A e^{-f_A(y_0)} \Log y_0.
\]
We claim that this implies the existence of an absolute constant $c>0$ such that 
\begin{equation}
	\label{eq:tau(n_{<y}) for y in [y_0,y_0^2]}
		\tau(n_{<y}) \ge cA e^{-f_A(y)} \Log y
		\qquad\text{for all}\ y\in[y_0,y_0^2].
\end{equation}
Indeed, if $\Log_2y_0\ge 10\Log A$, then $f_A(y)=\delta(\Log A+\Log_2y)$ for all $y\in[y_0,y_0^2]$, so \eqref{eq:tau(n_{<y}) for y in [y_0,y_0^2]} holds for some appropriate choice of $c>0$; on the other hand, if $\Log_2 y_0\le 10\Log A$, then both functions in the right-hand side of \eqref{fadef} change by at most $O(1)$ when $y$ ranges in $[y_0,y_0^2]$, so \eqref{eq:tau(n_{<y}) for y in [y_0,y_0^2]} holds again provided we choose $c>0$ to be small enough. 

Now, using \eqref{eq:tau(n_{<y}) for y in [y_0,y_0^2]}, we find that
\[
\int_1^{x^2} 1_{\tau(n_{<y}) \ge cA e^{-f_A(y)} \Log y} \frac{\dee y}{y \Log y} \gg 1.
\]
We conclude that
\[
\sum_{n \in \CS_{<x} \backslash \CS_{<x}^A} \frac{1}{n} 
	\ll \int_1^{x^2} \sum_{n \in \CS_{<x}} \frac{1_{\tau(n_{<y}) \ge cA e^{-f_A(y)} \Log y}}{n} 
	\cdot \frac{\dee y}{y \Log y}.
\]
Factoring $n = n_{<y} n_{\geq y}$ and using Mertens' theorem we have
\[
\sum_{n \in \CS_{<x}} \frac{1_{\tau(n_{<y}) \ge cA e^{-f_A(y)} \Log y}}{n} 
\asymp \frac{\Log x}{\Log y} \sum_{n \in \CS_{<y}} \frac{1_{\tau(n) \ge cA e^{-f_A(y)} \Log y}}{n} ,
\]
so it suffices to show that
\begin{equation}\label{mango}
\int_1^{x^2} \sum_{n \in \CS_{<y}} \frac{1_{\tau(n) 
			\ge cA e^{-f_A(y)} \Log y}}{n}\cdot \frac{\dee y}{y \Log^2 y} \ll \frac{1}{A}.
\end{equation}

First, we dispose of some easy contributions.  If $\Log y \leq A^{0.01}$, then we bound
\[
\sum_{n \in \CS_{<y}} \frac{1_{\tau(n) \ge cA e^{-f_A(y)} \Log y}}{n}
\leq \frac{1}{(cA e^{-f_A(y)} \Log y)^2} \sum_{n \in \CS_{<y}} \frac{\tau(n)^2}{n} 
\ll \frac{\Log^2 y}{A^2} e^{2f_A(y)}
\]
by Lemma \ref{mertens}, and the contribution of this case to the left-hand side of \eqref{mango} is easily seen to be acceptable for $\delta\le1/3$, which we may assume.

In the other extreme, if $\Log y \geq A^{100}$, then we bound
\[
\sum_{n \in \CS_{<y}} \frac{1_{\tau(n) \ge cA e^{-f_A(y)} \Log y}}{n}
\leq \frac{1}{(cA e^{-f_A(y)} \Log y)^{1/2}} \sum_{n \in \CS_{<y}} \frac{\tau(n)^{1/2}}{n} 
\ll \frac{(\Log y)^{\sqrt{2}-1/2}}{A^{1/2}} e^{f_A(y)/2}
\]
using Lemma \ref{mertens} again, and one can check here too that this contribution to the left-hand side of \eqref{mango} is acceptable if $\delta\le1/20$, which we may assume.

In conclusion, in order to prove \eqref{mango}, it will suffice to establish a bound of the form
\begin{equation}\label{orange} 
\sum_{n \in \CS_{<y}} \frac{1_{\tau(n) \ge cA e^{-f_A(y)} \Log y}}{n} \ll \frac{e^{-f_A(y)}}{A (\Log A)^{1/2}} \Log y
\end{equation}
whenever $A^{0.01} \le \Log y \le A^{100}$. This essentially follows by work of Norton \cite{norton} (see also \cite[Theorems 08 and 09]{HT-book}). We give the details below.

We have $\tau(n) = 2^{\omega(n)}$, and thus $\tau(n)\ge cA e^{-f_A(y)}\Log y$ if, and only if, 
\[
\omega(n)\ge k_y \coloneqq \fl{\frac{\log c+\log A - f_A(y) +\log(\Log y)}{\log 2}}  .
\]
In addition, for each $k\in\Z_{\ge0}$ we have
\[
\sum_{\substack{n \in \CS_{<y} \\ \omega(n) = k}} \frac{1}{n} \leq \frac{1}{k!} \bggg(\sum_{p<y} \frac{1}{p}\bggg)^k 
\leq \frac{(\Log_2 y + C)^k}{k!} 
\]
for some constant $C>0$, by Mertens' theorem \cite[Theorem 3.4(b)]{dk-book}. 
Notice that  $k_y \geq 1.1 (\Log_2 y+C)$, which implies that the quantities $\frac{1}{k!} (\Log_2 y + C)^k$ decay at least exponentially fast for $k\ge k_y$. We thus conclude that
\[ 
\sum_{n \in \CS_{<y}} \frac{1_{\tau(n) \ge cA e^{-f_A(y)} \Log y}}{n}
	\le \sum_{k\ge k_y} \frac{(\Log_2 y + C)^k}{k!} 
	\ll \frac{(\Log_2 y + C)^{k_y}}{k_y!} .
\]
By Stirling's formula and the bounds $k_y \asymp \Log_2 y \asymp \Log A$, we then have
\begin{equation}
	\label{eq:stirling for divisor bound}
	\sum_{n \in \CS_{<y}} \frac{1_{\tau(n) \ge cA e^{-f_A(y)} \Log y}}{n}
	\ll \frac{(\Log y)^{1-Q(t_y)}}{(\Log A)^{1/2}},
\end{equation}
where
\[
Q(t)=t\log t-t+1
\quad\text{and}\quad
t_y= \frac{k_y}{\Log_2y+C} =  \frac{\Log A-f_A(y)+\Log_2y}{(\log 2) \Log_2y} + O\bgg(\frac{1}{\Log_2y}\bgg) .
\]
Observe that $t_y\in[1.1,150]$ when $A^{0.01}\le\Log y\le A^{100}$, $\delta\le 1/5$ and $A$ is large enough.

Now, note that 
\begin{equation}
	\label{eq:t-e1}
	t_y-2 = \frac{\Log A-(\log4-1)\Log_2y - f_A(y)}{(\log 2)\Log_2y} + O\bgg(\frac{1}{\Log_2y}\bgg) .
\end{equation}
In addition, we have $0\le f_A(y)\le 100\delta|\Log_2y-\frac{\Log A}{\log 4-1}|$, and thus
\begin{equation}
	\label{eq:t-e2}
	\frac{|\Log_2y - \frac{\Log A}{\log4-1}|}{2\Log_2y} \le |t_y-2| \le \frac{|\Log_2y - \frac{\Log A}{\log4-1}|}{\Log_2y}. 
\end{equation}
if $\delta$ is small enough and $A$ is large enough. We shall now use Taylor's theorem to approximate $Q(t_y)$ by $Q(2)$. Since $t_y\in[1.1,150]$, there must exist some $\xi\in[1.1,150]$ such that 
\[
Q(t_y)=  Q(2)+Q'(2) (t_y-2) + Q''(\xi) \frac{(t_y-2)^2}{2} .
\]
We have $Q(2)=\log4-1$, $Q'(2)=\log2$ and $Q''(\xi)=1/\xi\ge 1/150$. We then use \eqref{eq:t-e2} to obtain a lower bound on $(t_y-2)^2$, and subsequently \eqref{eq:t-e1} to estimate $t_y-2$. In conclusion, we have
\begin{align*}
		Q(t_y)\Log_2y &\ge (\log4-1)\Log_2y + (t_y-2) (\log2)\Log_2y + 2f_A(y) \\
		&= \Log A+f_A(y)+O(1),
\end{align*}
as long as $\delta$ is small enough. Inserting this estimate into \eqref{eq:stirling for divisor bound} completes the proof of \eqref{orange}, and thus of the proposition.
\end{proof}

\section{The key moment estimate}

For inductive purposes we will need to introduce a quantity $m_{q,A}$ depending on several parameters $C_0, A, q$. According to these quantities, we shall then define $\CS_{<x}^{q,A}$ to be the set of all integers $n\in\CS_{<x}^A$ such that 
\begin{equation}\label{matn}
	M_a(n)/\tau(n) \leq m_{a,A}
	\qquad\text{for all}\ a=1,2,\dots,q.
\end{equation}
Observe that $M_1(n)=\tau(n)$, and thus the above inequality is trivially satisfied when $a=1$ as long as we ensure that 
\[
m_{1,A}\ge1.
\]
In particular, 
\begin{equation}
	\label{eq:S when q=1}
	\CS_{<x}^{1,A}=\CS_{<x}^A.
\end{equation}

Clearly we have the inclusions
\[
\CS_{<x} \supset \CS^{1,A}_{<x} \supset \CS^{2,A}_{<x} \supset \dots.
\]
In addition, from \eqref{mq0} we have
\[
M_a(pn)/\tau(pn) \geq M_a(n)/\tau(n)
\]
whenever $p$ is a prime, $n$ is coprime to $p$, and $a \geq 1$.  In particular,  $M_a(n_{<y})/\tau(n_{<y})$ is a non-decreasing function of $y$, and thus 
\[
M_a(n_{<y})/\tau(n_{<y}) \leq m_{a,A}
	\qquad\text{for}\ a=1,2,\dots,q\ \text{and}\ y\in[1,x].
\]
In other words, we have that
\begin{equation}\label{csqy}
	n_{<y} \in \CS^{q,A}_{<y}
	\quad\text{whevever}\ n \in \CS^{q,A}_{<x}\ \text{and}\ y\in[1,x] .
\end{equation}

We shall choose
\begin{equation}\label{mq-def}
	m_{q,A} \coloneqq \frac{q!}{q^2} (C_0 A)^{q-1} (\Log A)^{\frac{1}{2}(q-1+\lfloor q/2 \rfloor )},
\end{equation}
where $C_0$ is a large enough constant to be determined. We now show that our choice satisfies certain properties.

\begin{lemma}[The recursive upper bound]\label{recurse}  The following properties hold, with all implied constants independent of $q,A$ and $C_0$:
\begin{itemize}
\item[(i)]  One has $m_{1,A} \geq 1$, $m_{2,A} \gg A \Log A$, and $m_{q,A} \gg (C_0 A/3)^{q-1} q^q$. 
\item[(ii)]  For any $q \geq 3$, one has
\[
\sum_{\substack{a+b=q \\ 1 \le b \le q/2}} \binom{q}{a} m_{b,A} m_{a,A} \ll \frac{1}{C_0 A (\Log A)^{1/2}} \cdot m_{q,A}.
\]
\item[(iii)]  For any $q \geq 1$, one has
\[
(A m_{q,A})^{1/q} \ll q C_0 A (\Log A)^{3/4}.
\]
\end{itemize}
\end{lemma}

\begin{proof} The claims (i) and (iii) are clear from \eqref{mq-def} (bounding $q! \leq q^q$ and $q-1+\fl{q/2}\le 3q/2$).  For (ii), we calculate
\[
\binom{q}{a} m_{b,A} m_{a,A} 
	= \frac{q!}{a^2 b^2} (C_0 A)^{a+b-2} (\Log A)^{\frac{1}{2}(a+b-2+\fl{a/2}+\fl{b/2})} .
\]
Noticing that $a+b=q$, $\fl{a/2}+\fl{b/2}\le \fl{q/2}$, and $a^2 \asymp q^2$, the claim follows from the summability of $\sum_{b=1}^\infty \frac{1}{b^2}$.
\end{proof}

We now prove the following key moment estimate. In its proof, we shall only use the three properties of the parameters $m_{q,A}$ given in Lemma \ref{recurse}. We may thus think of these properties as the only axioms our parameters need to satisfy.

\begin{proposition}[Key moment estimate]\label{iterative}  Suppose that $C_0 \geq 1$ is a sufficiently large constant, and $A \geq 1$.  Then for any $q \geq 2$ and $x > 1$ we have the bound
\begin{equation}\label{targ}
 \sum_{n \in \CS^{q-1,A}_{<x}} \frac{M_q(n)/\tau(n)}{n} \leq \frac{C_0}{q^2 A} m_{q,A} \Log x.
\end{equation}
\end{proposition}

\begin{proof}  We induct on $q$, assuming that the claim has already been proven for all smaller values of $q$ (this assumption is vacuous for $q=2$).  We fix $A$ and introduce the notation
\[
T_q(x) \coloneqq \sum_{n \in \CS^{q-1,A}_{<x}} \frac{M_q(n)/\tau(n)}{n}.
\]
Every natural number $n \in \CS^{q-1,A}_{<x}$ other than $1$ is expressible in the form $n = pm$ with $p < x$ a prime and $m \in \CS^{q-1,A}_{<p}$ (here we use \eqref{csqy}).  Thus
\[
T_q(x) \leq 1 + \sum_{p<x} \sum_{n \in \CS^{q-1,A}_{<p}} \frac{M_q(pn)/\tau(pn)}{pn}.
\]
Applying \eqref{mq-2}, we conclude that
\[
T_q(x) \leq \sum_{p<x} \frac{T_q(p)}{p} + Q_q(x),
\]
where
\begin{equation}
	\label{eq:Qq dfn}
	Q_q(x) \coloneqq 1 + \sum_{p<x} \sum_{n \in \CS^{q-1,A}_{<p}} \frac{1}{\tau(n) pn}
	\sum_{\substack{a+b=q \\ 1 \leq b \leq q/2}} \binom{q}{a} \int_\R \Delta(n;u)^a \Delta(n;u-\log p)^b\dee u.
\end{equation}
We can iterate this inequality in the obvious fashion to arrive at
\[
T_q(x) \leq Q_q(x) + \sum_{\substack{ n \in \CS_{<x}\\ n > 1} } \frac{Q_q(P^-(n))}{n} , 
\]
where $P^-(n)$ is the least prime factor of $n$ with the convention that $P^-(1)=+\infty$.  Note that
\[
\sum_{\substack{n \in \CS_{<x} \\  P^-(n)=p_0}} \frac{1}{n} 
	= \frac{1}{p_0} \prod_{p_0 < p < x} \bggg(1+\frac{1}{p}\bggg) \asymp \frac{1}{p_0} \cdot\frac{\Log x}{\Log p_0}
\]
for any prime $p_0 < x$, and thus
\begin{equation}\label{tax}
 T_q(x) \ll Q_q(x) + \sum_{p < x} \frac{Q_q(p)}{p} \cdot \frac{\Log x}{\Log p}.
\end{equation}

We now turn to the estimation of $Q_q(x)$. Recall its definition in \eqref{eq:Qq dfn}. Note that if $n \in \CS_{<p}^{q-1,A}$, then $n \in \CS_{<y}^{q-1,A}$ for all $y \in [p,p^2]$ because $n_{<y}=n_{<p}$ for all such values of $y$ and the function $w\to e^{-f_A(w)}\Log w$ is increasing. Since $\int_p^{p^2}\dee y/(y\Log y)\asymp1$, we conclude that
\begin{align*}
	Q_q(x) &\ll 1 + \sum_{p<x} \int_{p}^{p^2}\sum_{n \in \CS^{q-1,A}_{<y}} \frac{1}{\tau(n) pn}  
	\sum_{\substack{a+b=q \\ 1 \leq b \leq q/2}} \binom{q}{a} \int_\R \Delta(n;u)^a \Delta(n;u-\log p)^b\dee u  \frac{\dee y}{y\Log y} \\
	&\le 1 + \int_1^{x^2} \int_\R\sum_{\substack{a+b=q \\ 1 \leq b \leq q/2}}  \binom{q}{a}\sum_{n \in \CS^{q-1,A}_{<y}} \sum_{p \geq y^{1/2}} \frac{1}{\tau(n) pn}  \Delta(n;u)^a \Delta(n;u-\log p)^b\dee u \frac{\dee y}{y \Log y}.
\end{align*}

From \eqref{deltanu-def} followed by Lemma \ref{bt} we have
\begin{align*} 
\sum_{p \geq y^{1/2}} \frac{1}{p} \Delta(n;u-\log p)^b
&= \sum_{p \geq y^{1/2}}\frac{1}{p} \sum_{\substack{d_1,\dots,d_b|n \\ u-\log p < \log d_1,\dots,\log d_b \leq u-\log p+1 }} 1 \\
&= \sum_{\substack{d_1,\dots,d_b|n \\ \log d_{\max} < \log d_{\min}+1}} \sum_{\substack{p \geq y^{1/2} \\ u-\log d_{\min} < \log p \leq u -\log  d_{\max} + 1}} \frac{1}{p} \\
&\ll \sum_{\substack{d_1,\dots,d_b|n \\ \log d_{\max} < \log d_{\min}+1}}  \left(\frac{\log d_{\min} + 1 -\log  d_{\max}}{\Log y} + \frac{1}{y^{1/4}}\right),
\end{align*}
where we adopt the shorthand $d_{\min} \coloneqq \min(d_1,\dots,d_b)$ and $d_{\max} \coloneqq \max(d_1,\dots,d_b)$.  A similar computation also gives
\begin{align*}
 M_b(n) = \sum_{\substack{d_1,\dots,d_b|n \\ \log d_{\max} < \log d_{\min}+1}} \int_{u < \log d_1,\dots,\log d_b \leq u+1}\dee u 
 	=  \sum_{\substack{d_1,\dots,d_b|n \\ \log d_{\max} < \log d_{\min}+1}} (\log d_{\min} + 1 - \log d_{\max}),
\end{align*}
while
\begin{align}
\sum_{\substack{d_1,\dots,d_b|n \\ \log d_{\max} < \log d_{\min}+1}} 1
&\leq \sum_{\substack{d_1,\dots,d_b|n \\ \log d_{\max} < \log d_{\min}+2}} (\log d_{\min} + 2 - \log d_{\max}) \nn
&= \int_\R (\Delta(n;u) + \Delta(n;u+1))^b\ du \nn
&\leq 2^b M_b(n) \label{eq:MT ineq}
\end{align}
thanks to the triangle inequality in $L^b$ (the proof of inequality \eqref{eq:MT ineq} goes back to Maier and Tenenbaum \cite{MT2}). Combining all these estimates, we obtain the bound
\begin{equation}\label{qx}
 Q_q(x) \ll 1 + \int_1^{x^2} \sum_{\substack{a+b=q \\ 1 \leq b \leq q/2}} \binom{q}{a} \sum_{n \in \CS^{q-1,A}_{<y}} \left( \frac{1}{\Log y} + \frac{2^b}{y^{1/4}} \right) \frac{M_a(n) M_b(n)}{\tau(n) n} \cdot \frac{\dee y}{y \Log y}.
\end{equation}

At this point we split our analysis into the base case $q=2$ and the inductive case $q>2$.  

\medskip

{\it Base case $q=2$.} We must then have $a=b=1$.  Since $M_1(n) = \tau(n)$ and $\CS_{<x}^{1,A}=\CS_{<x}^A$ (cf.~\eqref{eq:S when q=1}), the bound \eqref{qx} simplifies to
\[
Q_2(x) \ll 1 + \int_1^{x^2} \sum_{n \in \CS^A_{<y}} \frac{\tau(n)}{n} \cdot \frac{\dee y}{y\Log^2 y}.
\]
On the one hand, we have from Mertens' theorem that
\[
\sum_{n \in \CS^A_{<y}} \frac{\tau(n)}{n} \leq \prod_{p<y} \bggg(1+\frac{2}{p}\bggg) \ll \Log^2 y	.
\]
On the other hand, from \eqref{mjn} and Lemma \ref{mertens} one has
\[
\sum_{n \in \CS^A_{<y}} \frac{\tau(n)}{n} 
	\leq \bg( A e^{-f_A(y)} \Log y\bg)^{1/2}  \sum_{n \in \CS_{<y}} \frac{\tau(n)^{1/2}}{n} 
	\ll A^{1/2} (\Log y)^{1/2+\sqrt{2}}.
\]
Consequently,
\[
Q_2(x) \ll 1 + \int_1^{x^2} \min\bg\{ A^{1/2} (\Log y)^{-0.01} , 1 \bg\} \frac{\dee y}{y} 
 	\ll \min\bg\{A^{1/2}(\Log x)^{0.99}, \Log x\bg\},
\] 
and thus by \eqref{tax}
\[
T_2(x) \ll \min\bg\{A^{1/2}(\Log x)^{0.99}, \Log x\bg\}
	+ \sum_{p < x} \frac{\min\bg\{A^{1/2}(\Log p)^{0.99}, \Log p\bg\}}{p} \cdot \frac{\Log x}{\Log p}.
\]
Dividing the summation into the ranges $\Log p \leq A^{50}$ and $\Log p > A^{50}$, and using Mertens' theorem, we conclude that
\[
T_2(x) \ll (\Log A) (\Log x) \ll \frac{1}{A} \cdot m_{2,A} \Log x
\]
thanks to Lemma \ref{recurse}(ii). Thus the claim \eqref{targ} follows for $C_0$ large enough. This concludes the treatment of the base case $q=2$.

\medskip

{\it Inductive case $q>2$.} We first handle the lower order term
\[
R_q(x)\coloneqq \int_1^{x^2} \sum_{\substack{a+b=q \\ 1 \leq b \leq q/2}} \binom{q}{a} \sum_{n \in \CS^{q-1,A}_{<y}} \frac{2^b}{y^{1/4}} \cdot \frac{M_a(n) M_b(n)}{\tau(n) n} \cdot \frac{\dee y}{y \Log y}
\]
appearing in \eqref{qx}.  We crudely use H\"older's inequality to bound
\[
M_a(n) M_b(n) \leq M_1(n) M_{q-1}(n) \le \tau(n)^q (1+\log n) .
\]
Since we also have $\sum_{a+b=q} \binom{q}{a} 2^b = 3^q$, we conclude that
\[
R_q(x) \le 3^q \int_1^{x^2} \sum_{n \in \CS^{q-1,A}_{<y}} \frac{\tau(n)^{q-1}(1+\log n)}{n} 
	\cdot \frac{\dee y}{y^{5/4} \Log y}.
\]
From \eqref{mjn} we have
\[
\tau(n)^{q-1} \leq (A\Log y)^{q-2} \tau(n),
\]
while 
\begin{align*}
\sum_{n \in \CS_{<y}} \frac{\tau(n)(1+\log n)}{n} &\le \bigg(1+2\sum_{p<y} \frac{\log p}{p} \bigg)  \prod_{p<y}\bggg(1+\frac{2}{p}\bggg) \ll (\Log y)^3.
\end{align*}
Thus
\[
R_q(x) \ll 3^q A^{q-2} \int_1^\infty \frac{(\Log y)^q \dee y}{y^{5/4}}
= 3^q A^{q-2} \cdot 4^{q+1} q! \le 12^{q+1} q^q A^{q-2} ,
\]
as can be seen by the change of variables $y=e^{4u}$. Inserting this into \eqref{qx} we conclude that
\begin{equation}
	\label{eq:bound for Rq}
	Q_q(x) \ll 12^q q^q A^{q-2} + Q'_q(x),
\end{equation}
where
\[
Q_q'(x)\coloneqq \int_1^{x^2} \sum_{\substack{a+b=q \\ 1 \leq b \leq q/2}} \binom{q}{a} \sum_{n \in \CS^{q-1,A}_{<y}} \frac{M_a(n) M_b(n)}{\tau(n) n} \cdot \frac{\dee y}{y \Log^2 y}.
\]
Applying successively \eqref{matn} and \eqref{mjn}, we find that
\[
M_b(n) \leq m_{b,A} A e^{-f_A(y)} \Log y ,
\]
and thus
\[
 Q_q'(x) \le \int_1^{x^2} \sum_{\substack{a+b=q \\ 1 \leq b \leq q/2}} \binom{q}{a} m_{b,A} A e^{-f_A(y)} T_a(y) \frac{\dee y}{y \Log y}.
 \]
Since $q > 2$, $a+b=q$, and $1 \leq b \leq q/2$, we have $2 \leq a < q$, and hence by induction hypothesis
\[
T_a(y) \leq \frac{C_0}{a^2 A} m_{a,A} \Log y.
\]
Since $a\ge q/2$, we have $a^2\ge q^2/4$. As a consequence,
\[
 Q_q'(x) \le \frac{4C_0}{q^2} \int_1^{x^2} \sum_{\substack{a+b=q \\ 1 \leq b \leq q/2}} \binom{q}{a} m_{a,A} m_{b,A} e^{-f_A(y)} \frac{\dee y}{y}
 \]
and hence by Lemma \ref{recurse}(ii)
\[
 Q_q'(x) \ll \frac{m_{q,A}}{q^2 A (\Log A)^{1/2}} \int_1^{x^2} e^{-f_A(y)} \frac{\dee y}{y}.
 \]
 
We make the change of variables $y=e^{e^t}$ to find that
\[
\int_1^{x^2} e^{-f_A(y)} \frac{\dee y}{y} \le \int_{-\infty}^{\Log_2x+1} e^{t-f_A(\exp\exp (t))}\dee t \ll e^{-f_A(x)} \Log x,
\]
where we used \eqref{fadef} with $\delta$ small enough to show that the function $t-f_A(\exp\exp(t))$ is piecewise differentiable with derivative bounded from below by an absolute positive constant.  
In conclusion,
\[
Q_q'(x) \ll \frac{e^{-f_A(x)}m_{q,A}\Log x}{q^2 A (\Log A)^{1/2}} .
\]
Together with \eqref{eq:bound for Rq}, this implies that
\[
Q_q(x)\ll 12^q q^q A^{q-2} + \frac{e^{-f_A(x)}m_{q,A}\Log x}{q^2 A (\Log A)^{1/2}} .
\]

Inserting the above bound into \eqref{tax}, and using Mertens' theorem, we conclude that
\begin{equation}
	\label{eq:Tq almost there}
	T_q(x) \ll 12^q q^q A^{q-2}  \Log x + \frac{m_{q,A} \Log x}{q^2 A (\Log A)^{1/2}} 
	\bigg(1 + \sum_{p} \frac{e^{-f_A(p)}}{p}\bigg) ,
\end{equation}
where we used that the sum $\sum_p \frac{1}{p\log p}$ converges. Finally, we break up the sum $\sum_p \frac{e^{-f_A(p)}}{p}$ over $p$ on the right-hand side of \eqref{eq:Tq almost there} into intervals such that $j \leq \Log_2 p < j+1$ for some $j\in\Z_{\ge0}$. For each fixed $j$, we have $f_A(p)=f_A(\exp\exp(j)) +O(1)$ as well as $\sum_{j\le \Log_2p<j+1} \frac{1}{p} \ll 1$ by Mertens' theorem. Consequently,
\[
\sum_{p} \frac{e^{-f_A(p)}}{p} 
	\ll \sum_{j\ge1} e^{-f_A(\exp\exp(j))} \ll (\Log A)^{1/2} ,
\]
by the definition of $f_A$ (cf.~\eqref{fadef}). Hence, using Lemma \ref{recurse}(i) we conclude (for $C_0$ large enough) that
\[
T_q(x) \leq \frac{C_0}{q^2 A} m_{q,A} \Log x.
\]
This completes the proof of the proposition.
\end{proof}

\section{Closing the argument}

Henceforth we fix $C_0$ so that Proposition \ref{iterative} applies, and allow implied constants to depend on $C_0$.

\begin{corollary}[Weak type estimate] Uniformly for $\lambda\ge1$, we have
\[
\sum_{\substack{ n \in \CS_{<x} \\  \Delta(n) \geq \lambda\Log_2x}} \frac{1}{n} 
	\ll \frac{(\Log \lambda)^{3/4}}{\lambda} \cdot \Log x.
\]
\end{corollary}

\begin{proof}  Let $C_1$ be a large constant and define $A>0$ implicitly via the equation 
	\[
	\lambda = C_1 A (\Log A)^{3/4}.
	\]
We may assume that $A \geq 1$, as the estimate is trivial otherwise. 
Our task is now to show that
\[
\sum_{\substack{ n \in \CS_{<x} \\ \Delta(n) \geq \lambda\Log_2x}} 
	\frac{1}{n} \ll \frac{\Log x}{A}.
\]
From Proposition \ref{gauss} and relation \eqref{eq:S when q=1}, we have
\begin{equation}
	\label{eq:gauss again}
		\sum_{n \in \CS_{<x} \backslash \CS_{<x}^{1,A}} \frac{1}{n} \ll \frac{\Log x}{A}.
\end{equation}
Also, from \eqref{matn}, Proposition \ref{iterative}, and Markov's inequality, we have for all $j \geq 2$ that
\begin{equation}
	\label{eq:S_{j-1}-S_j}
\sum_{n \in \CS^{j-1,A}_{<x} \backslash \CS_{<x}^{j,A}} \frac{1}{n} 
		\leq \frac{1}{m_{j,A}} \sum_{n \in \CS^{j-1,A}_{<x}} \frac{M_j(n)/\tau(n)}{n} 
		\ll \frac{\Log x}{j^2 A}.
\end{equation}
Summing \eqref{eq:gauss again} and \eqref{eq:S_{j-1}-S_j} for $j=2,\dots,q$, we conclude that
\[
\sum_{n \in \CS_{<x} \backslash \CS_{<x}^{q,A}} \frac{1}{n} \ll \frac{\Log x}{A} 
\qquad\text{for all}\ q\in\N.
\]
The corollary will then follow if we can show that there exists $q\in\N$ such that
\begin{equation}
\label{eq:Delta control}
\Delta(n)<\lambda\Log_2x \quad\text{for all}\ n\in\CS_{<x}^{q,A}.
\end{equation}

Indeed, let us fix $q\in\N$ to be chosen later and let $n \in \CS_{<x}^{q,A}$. From Theorem 72 in \cite{HT-book}, we know that\footnote{For completeness, we give the short proof of this inequality. We have $\Delta(n) = \Delta(n;u_0)$ for some real $u_0$, hence $\Delta(n)^q \leq (\Delta(n;u) + \Delta(n;u+1))^q \leq 2^{q-1} ( \Delta(n;u)^q + \Delta(n;u+1)^q )$
	for all $u\in[u_0 - 1,u_0]$. Integrating both sides over $u\in[u_0-1,u_0]$ yields the inequality $\Delta(n)^q \leq 2^q M_q(n)$.}
\[
\Delta(n)^q \leq 2^q M_q(n).
\]
Hence, by \eqref{matn} and \eqref{mjn}, we have 
\[
\Delta(n)^q \ll 2^q A m_{q,A} \Log x.
\]
Taking $q^{\mathrm{th}}$ roots and using Lemma \ref{recurse}(iii), we find that
\[
\Delta(n) \ll q A (\Log A)^{3/4} (\Log x)^{1/q} .
\]
We take $q \coloneqq \lfloor \Log_2 x\rfloor$ to optimize constants. Recalling the definition of $A$ in terms of $\lambda$, and assuming the constant $C_1$ there is chosen to be large enough, we conclude that \eqref{eq:Delta control} does hold for all $n\in\CS_{<x}^{q,A}$. This completes the proof of the corollary.
\end{proof}

\begin{corollary}[Strong type estimate]\label{strong} For any $x \geq 1$, we have
\[
\sum_{n \in \CS_{<x}} \frac{\Delta(n)}{n} \ll (\Log_2 x)^{11/4} \Log x.
\]
\end{corollary}

\begin{proof}  For those $n$ with $\Delta(n) \geq (\Log x)^{10}$, we use the trivial bound $\Delta(n) \leq \tau(n)^2/(\Log x)^{10}$, and this contribution is acceptable by Lemma \ref{mertens}. 

On the other hand, those $n$ with $\Delta(n)\le \Log_2x$ also have an acceptable contribution because $11/4>1$. 
	
We then subdivide the remaining range $\Log_2x \leq \Delta(n) < (\Log x)^{10}$ into $O(\Log_2 x)$ dyadic ranges $2^j \Log_2x \leq \Delta(n) < 2^{j+1}\Log_2x $ with $j\in\Z_{\ge0}$. In each range we use Corollary \ref{strong}. Thus
\begin{align*}
		\sum_{\substack{n \in \CS_{<x} \\ \Log_2x\le\Delta(n)<(\Log x)^{10}}} \frac{\Delta(n)}{n}
		&\le \sum_{0\le j\ll\Log_2x} \sum_{\substack{n \in \CS_{<x} \\ 2^j \le \Delta(n)/\Log_2x<2^{j+1}}}  \frac{\Delta(n)}{n} \\
		&\le \sum_{0\le j\ll\Log_2x} (2^{j+1}\Log_2x) \sum_{\substack{n \in \CS_{<x} \\ \Delta(n)\ge 2^j\Log_2x}} \frac{1}{n} \\
		&\ll \sum_{0\le j\ll\Log_2x} (2^{j+1}\Log_2x) \cdot  \frac{j^{3/4}}{2^j} \Log x \ll (\Log_2x)^{11/4}\Log x.
\end{align*}
This completes the proof.
\end{proof}

Lastly, Theorem \ref{main} follows immediately by Corollary \ref{strong} and inequality \eqref{eq:switching to log weights}.

\section{Proof of \eqref{eq:robert improvement}}
\label{sec:robert improvement} 

Fix $k,c_1,\dots,c_k,\ell_1,\dots,\ell_k$ as in Remark \ref{rem:diophantine}. All implied constants might depend on these parameters without further notice. 

Following the proof of Theorem 1.1 in Section 5 of \cite{robert}, we have
\begin{equation}
	\label{eq:robert e0}
		S^{\neq}(x) \ll x+ \frac{x}{\log x} (\Log_2x)^{2+2^{4L}} \sum_{p|m\ \Rightarrow\ p<y} \frac{\Delta(m)f(m)}{m}
\end{equation}
with $y=\exp(c\frac{\Log x}{\Log_2x})$ for some constant $c>0$ and $f(m)=N(\underline{\ell};\underline{c};m)/(m^{2k-2}\phi(m))$, where  $\phi(m)=\#(\Z/m\Z)^*$ is Euler's totient function and $N(\underline{\ell};\underline{c};m)$ is defined to be the number of tuples $(m_1,\dots,m_k,n_1,\dots,n_k)\in(\Z/m\Z)^{2k}$ such that $\sum_{j=1}^k c_jm_j^{\ell_j}\equiv \sum_{j=1}^k c_jn_j^{\ell_j}\mod m$. 

Now, in view of \cite[Lemma 3.4]{robert} and our assumption that\footnote{When $k=1$, we have $f(p)=2+O(1/p)$, and the behaviour of $\sum_{p|m\ \Rightarrow p<y} f(m)\Delta(m)/m$ changes. Indeed, the case $k=1$ of \eqref{eq:robert diophantine equation} corresponds to the classical problem of which integers $n$ can be written in the form $c_0m_0^2+c_1m_1^2$. In particular, a correction is needed in \cite[Theorem 1.1]{robert} to indicate that $k$ must be at least $2$.} $k\ge2$, we have $f(p)=1+O(1/p)$ and $f(p^\nu)\le \nu^{O(1)}$ for $\nu\ge2$. Therefore
\begin{align}
		\sum_{p|m\ \Rightarrow\ p<y} \frac{\Delta(m)f(m)}{m}
			&\ll \sum_{m\in \CS_{<y}}  \frac{\Delta(m)f(m)}{m}  \label{eq:robert e1}\\
			&\ll  \sum_{m\in \CS_{<y}}  \frac{\Delta(m)}{m}  \label{eq:robert e2}\\
			&\ll (\Log y)(\Log_2y)^{11/4} \asymp (\Log x)(\Log_2x)^{7/4}, \label{eq:robert e3}
\end{align}
where \eqref{eq:robert e1} is proven by writing $m=m_1m_2$ with $m_1$ square-free, $m_2$ square-full and $(m_1,m_2)=1$, so that $\Delta(m)\le \Delta(m_1)\tau(m_2)$, \eqref{eq:robert e2} is proven by writing $f=1*g$ so that $f(m)\Delta(m)\le \sum_{ab=m}\Delta(a)|g(b)|\tau(b)$ for $m$ square-free (because we must then have $(a,b)=1$ whenever $m=ab$, and thus $\Delta(m)\le \Delta(a)\tau(b)$), and \eqref{eq:robert e3} follows by Corollary \ref{strong} and the definition of $y$.

 Combining \eqref{eq:robert e0} and \eqref{eq:robert e3} completes the proof of \eqref{eq:robert improvement}.

\end{document}